\documentclass[final,3p,times]{elsarticle}
\usepackage{graphics}
\usepackage[utf8]{inputenc} 
\usepackage{latexsym,amssymb,amsthm,amsmath,amsfonts}
\journal{define}
\usepackage{url}
  \newtheorem{theorem}{Theorem}
  \newtheorem{proposition}{Proposition}
               \newtheorem{corollary}{Corollary}
  \newtheorem{example}{Example}

               \def\pf{\par\noindent {\em Proof.}~\par\noindent}
               \def\qed{~\hfill{$\square$}\pagebreak[1]\par\medskip\par}

\newcommand{\Li}{{\mbox{Lip}}}

\newcommand{\R}{{\mathbb R}}

\newcommand{\dist}{{\rm dist}}

\newcommand{\ux}{\underline{x}}

\newcommand{\uy}{\underline{y}}

\newcommand{\pux}{\partial_{\ux}}
\newcommand{\puy}{\partial_{\uy}}
\newcommand{\pD}{{^\psi\!\pux}}
\newcommand{\pDy}{{^\psi\!\puy}}
\newcommand{\hD}{{^\varphi\!\pux}}

\newcommand{\La}{{^{\varphi,\psi}\!\mathcal{L}_{\underline{x}}}}
\newcommand{\Lay}{{^{\varphi,\psi}\!\mathcal{L}_{\underline{y}}}}

\newcommand{\bj}{{\bf j}}

\begin{document}

\begin{frontmatter}

\title{Jump problem for generalized Lam\'e-Navier systems in $\R^m$}

\author{Daniel Alfonso Santiesteban$^{(1)}$, Ricardo Abreu Blaya$^{(2)}$ and Daniel Alpay$^{(3)}$}

\address{
$^{1,2}$Facultad de Matem\'aticas, Universidad Aut\'onoma de Guerrero, Mexico\\$^{3}$ Schmid College of Science and Technology, Chapman University, USA}
\ead{danielalfonso950105@gmail.com,  rabreublaya@yahoo.es, alpay@chapman.edu}

\begin{abstract}
This paper is devoted to study a fundamental system of equations in Linear Elasticity Theory: the famous Lam\'e-Navier system. The Clifford algebra language allows us to rewrite this system in terms of the Euclidean Dirac operator, which at the same time suggests a very natural generalization involving the so-called structural sets. Our interest lies mainly in the jump problem for these elastic systems. A generalized Teodorescu transform, to be introduced here, provides the means for obtaining the explicit solution of the jump problem for a very wide classes of regions, including those with a fractal boundary. 
 \end{abstract}

\end{frontmatter}
\section{Introduction}

Clifford analysis focuses on the study of the null solutions of the Dirac operator 
\begin{equation}\label{Diracoperator}
\pux=\sum_{j=1}^me_j\partial_{x_j}
\end{equation}
in an open region of $\R^m$. These solutions are known as monogenic functions and represent genuine generalizations of holomorphic functions of the complex plane. Here $(e_1,...,e_m)$ is an orthonormal basis in $\R^m$ underlying the construction of the $2^m$-dimensional universal real Clifford algebra $\R_{0,m}$, and $\ux=\sum_{j=1}^mx_je_j$ is a vector variable defined in $\R^m$. For an in-depth study about Clifford analysis, including historical notes, we refer to e.g. \cite{BDS,DSS,Delang,DK,GHS,Ryan,Ryan1}. 

Clifford algebras have essential applications in various fields of research, including theoretical physics, geometry, $K$-theory, cohomology, computer vision, robotics, quantum mechanics and digital image processing. These algebras are fundamental for describing symmetries in spacetime, especially in relation to rotations and Lorentz transformations. Clifford algebras provide a mathematical structure that unifies the inner and outer products of vectors, allowing for an elegant treatment of geometric concepts. 

Let $\{e_1,...,e_m\}$ be the standard basis of the inner product vector space $\R^{0,m}$ equipped with a
geometric product according to the multiplication rules
\begin{equation}
e_ie_j+e_je_i=-2\delta_{i,j},\quad i,j=1,...,m,
\end{equation}
where $\delta_{i,j}$ is the Kronecker symbol. This product and the additional axiom of associativity generate the $2^m$-dimensional real Clifford algebra $\R_{0,m}$. We shall freely use the well-known properties of Clifford algebras which the reader can find in many sources. Nevertheless, we compile some of those relevant to our purpose for completeness. The algebra $\R_{0,m}$ can be split up subspace of $k$-vectors $(k\in\mathbb N_m\cup\{0\})$
$$\R_{0,m}^{(k)}=\left\{a\in\R_{0,m}:a=\sum_{|A|=k}a_Ae_A,a_A\in\R\right\},$$
where $A$ runs over all the possible ordered sets $A=\{1\leq j_1<...<j_k\leq m\}$, or $A=\emptyset$, and $e_A:=e_{j_1}e_{j_2}...e_{j_k}$, $e_0=e_{\emptyset}=1$. Namely, $\R_{0,m}=\bigoplus_{k=0}^m\R_{0,m}^{(k)}$. The general elements of $\R_{0,m}$ are real linear combinations of basis blades $e_A$, called Clifford numbers, multivectors, or hypercomplex numbers. Any Clifford number $a\in\R_{0,m}$ may thus be written as
\begin{equation}
a=[a]_0+[a]_1+...+[a]_k+...+[a]_m,
\end{equation}
where $[]_k$ is the linear projection of $a$ onto the subspace of $k$-vectors $\R_{0,m}^{(k)}$. In particular, $\R_{0,m}^{(1)}$, $\R_{0,m}^{(0)}\oplus\R_{0,m}^{(1)}$ and $\R_{0,m}^{(m)}$ are called, respectively, the space of vectors, paravectors and pseudoscalars in $\R_{0,m}$. Note that $\R^{m}$ may be naturally identified with $\R_{0,m}^{(1)}$ by associating to any element $(x_1,...,x_m)\in\R^{m}$ the vector $\ux=\sum_{j=1}^mx_je_j$. For a vector $\vec{u}$ and a $k$-vector $F_k$, their product  $\vec{u}F_k$ splits into a $(k-1)$-vector and a $(k+1)$-vector, namely: 
$$\vec{u}F_k=[\vec{u}F_k]_{k-1}+[\vec{u}F_k]_{k+1},$$
where 
$$[\vec{u}F_k]_{k-1}=\frac{1}{2}[\vec{u}F_k-(-1)^kF_k\vec{u}]$$
and
$$[\vec{u}F_k]_{k+1}=\frac{1}{2}[\vec{u}F_k+(-1)^kF_k\vec{u}].$$
The inner and outer products between $\vec{u}$ and $F_k$ are defined by $\vec{u}\bullet F_k:=[\vec{u}F_k]_{k-1}$ and $\vec{u}\wedge F_k:=[\vec{u}F_k]_{k+1}$, respectively. In the same way, we define $\pux\bullet F_k=\frac{1}{2}[\pux F_k-(-1)^kF_k\pux]=[\pux F_k]_{k-1}$ and $\pux\wedge F_k=\frac{1}{2}[\pux F_k+(-1)^kF_k\pux]=[\pux F_k]_{k+1}$.

Conjugation in $\R_{0,m}$ is defined as the anti-involution $a\to\overline{a}$ for which $\overline{e_i}=-e_i$. Therefore,
$$\overline{a}=\sum_Aa_A\overline{e}_A,\quad\overline{e}_A=(-1)^{\frac{|A|(|A|+1)}{2}}e_A.$$
A norm $\|.\|$ on $\R_{0,m}$ is defined by $\|a\|^2=[a\overline{a}]_{0}=\sum_{A}a_A^2$.

In the sequel, when speaking of a domain $\Omega$, it will always be assumed to be open and simply connected set of $\R^m$ with sufficiently smooth boundary $\partial\Omega$. We will denote by $n(\ux)$ the outward unit normal vector at $\ux\in\partial\Omega$. In Section \ref{section3} we will consider domains with boundary satisfying conditions of a more general character. We shall use the temporary notation $\Omega_+=\Omega$, $\Omega_-=\R^m\setminus\overline{\Omega}$. We will consider functions $f:\Omega\to\R_{0,m}$ to be written as $f=\sum_{A}f_Ae_A$, where $f_A$ are $\R$-valued functions. The notions of continuity, differentiability and integrability of an $\R_{0,m}$-valued function have the usual component-wise meaning. The space of all $s$-times continuous differentiable, $s$-times $\nu$-H\"older continuously differentiable and $p$-integrable
functions are denoted by $C^s(\Omega)$, $C^{s,\nu}(\Omega)$ and $L^p(\Omega)$, respectively. 

The Dirac operator \eqref{Diracoperator} 
factorizes the Laplacian in $\R^{m}$ in the sense that $-\Delta=\pux\pux$. This fact implies that monogenic functions are also harmonic. In this way, Clifford analysis may be viewed as a refinement
of classical harmonic analysis. 

The sandwich equation 
\begin{equation}\label{eqsand}
\pux f\pux=0,
\end{equation}
can be seen as a non-commutative version of the more traditional Laplace equation
\begin{equation}\label{eqlaplace}
\pux\pux f=0.
\end{equation}
Solutions of \eqref{eqsand} are known as inframonogenic functions. Such functions were introduced in \cite{MPS2}, where a Fischer decomposition for homogeneous polynomials in terms of inframonogenic polynomials was established. 

Interest in this class of functions has grown significantly in recent years,  as evidenced by the increasing number of papers published on the subject, see for example  \cite{AbreuDAS,DAS1,DAS2,DAS3,DAS5,joao,Din,lavicka,MPS1,MAB1,MAB4,MMA,MDAS,wang}.
Using the inner and outer products, the equations \eqref{eqsand} and \eqref{eqlaplace} can be rewritten respectively as
\begin{equation}
\sum_{k=0}^m(-1)^k(\pux\bullet\pux\wedge[f]_k-\pux\wedge\pux\bullet[f]_k)=0
\end{equation}
and
\begin{equation}
\sum_{k=0}^m(\pux\bullet\pux\wedge[f]_k+\pux\wedge\pux\bullet[f]_k)=0.
\end{equation}
It should be noticed that when restricting our attention to vector-valued solutions $f =\vec{u}$ in $\R^3$, the above sandwich equation \eqref{eqsand} can be written in the form
\begin{equation}
\nabla(\nabla\cdot\vec{u})+\nabla\times\nabla\times\vec{u}=0,
\end{equation}
while Laplace equation takes the form
\begin{equation}
\nabla(\nabla\cdot\vec{u})-\nabla\times\nabla\times\vec{u}=0.
\end{equation}
That subtle change between sandwich and Laplace equations causes crucial differences in the properties of their respective solutions. The Laplacian is a strongly elliptic operator, while the operator $\pux(.)\pux$ is not. This means that we cannot, in general, ensure the uniqueness of the solution to the Dirichlet problem associated with \eqref{eqsand}.

The displacement vector $\vec{u}$ of the points of  a three-dimensional isotropic elastic body in the
absence of body forces is described by the Lam\'e-Navier system
\begin{equation}\label{sln}
\mu\Delta\vec{u}+(\mu+\lambda)\nabla(\nabla\cdot\vec{u})=0,
\end{equation}
where $\lambda$ and $\mu$ are the so-called Lam\'e parameters. Here, the quantities $\mu > 0$ and $\lambda >-\frac{2}{3} \mu$ are elastic constants that depend on the body. This system was originally introduced by the French mathematician Gabriel Lam\'e in 1837 while studying the method of separation of variables for solving the wave equation in elliptic coordinates \cite{lame}. In \cite{MAB3}, the authors reformulate the system \eqref{sln} in terms of the three-dimensional Dirac operator as follows
\begin{equation}\label{68}
\left(\frac{\mu+\lambda}{2}\right)\pux \vec{u}\pux+\left(\frac{3\mu+\lambda}{2}\right)\pux\pux \vec{u}=0,
\end{equation}
and proved that $\vec{u}$ can be decomposed as the sum of a harmonic vector field  and an inframonogenic vector field. Note in \eqref{68} that $\frac{3\mu+\lambda}{2}>\frac{\mu+\lambda}{2}$, which causes the strong ellipticity of the Laplacian to prevail in this equation. This system \eqref{68} has been studied recently in \cite{marco1,DASJD,marco2,MD}.

Of particular interest for our purpose is the paper \cite{DAS4}, where the authors study natural generalizations of the Lam\'e-Navier system using arbitrary orthonormal bases. In general, we will work with the following generalized Lam\'e-Navier system 
\begin{equation}\label{gLNs}
\left(\frac{\mu+\lambda}{2}\right)\hD f\pD+\left(\frac{3\mu+\lambda}{2}\right)\hD\pD f=0,
\end{equation}
where 
$$\hD=\sum_{j=1}^m\varphi_j\partial_{x_j}\quad\text{and}\quad
\pD=\sum_{j=1}^m\psi_j\partial_{x_j}$$
are the corresponding Dirac operators constructed with the orthonormal bases of $\R^m$: $\varphi=\{\varphi_1,...,\varphi_m\}$ and $\psi=\{\psi_1,...,\psi_m\}$. In the literature, these orthonormal bases are known as structural sets \cite{shapiro1,shapiro2}. We will denote by $\La$ the generalized Lam\'e-Navier operator: $\left(\frac{\mu+\lambda}{2}\right)\hD(.)\pD+\left(\frac{3\mu+\lambda}{2}\right)\hD\pD(.)$.

In \cite{GKS}, G\"urlebeck et al. showed that the class of $\psi$-hyperholomorphic functions (the solutions of the generalized Dirac equation $\pD f=0$) is more than what we get by orthogonal transformation, when the picture changed completely in the study of a $\Pi$-operator which involve a pair of different orthonormal basis. First of all there is no more a single orthogonal transformation which reduces it to the standard case. Second we have that for two different structural sets they factorize a second order operator which is not anymore a scalar operator. As it turns out several important questions are linked with and uniformized when
two different structural sets take part.

A characteristic feature of the generalized system \eqref{gLNs} is that the relative classical Dirichlet problem is no more well-posed, unlike what happens for the elastic equilibrium equation. However, in \cite{DAS6} polynomial solutions have been constructed for homogeneous Dirichlet problems associated with \eqref{gLNs} in ellipsoidal domains.

\section{Cauchy and Teodorescu transforms for $\La$}
Let be $\varphi=\{\varphi_1,...,\varphi_m\}$ a structural set in $\R^m$. In what follows we restrict our study to the case $m>2$. 

Since $\hD\hD=-\Delta$, the fundamental solution of $\hD$ (called Clifford-Cauchy kernel) is given by
$$K_\varphi(\ux)=\hD E_1(\ux)=-\frac{1}{\sigma_m}\frac{\ux_{\varphi}}{\|\ux\|^m}.$$
Here, $E_1$ is the fundamental solution of the Laplacian 
$$E_1(\ux)=\frac{1}{(m-2)\sigma_m\|\ux\|^{m-2}}$$
and $\ux_\varphi:=\sum_{j=1}^mx_j\varphi_j$ if $\ux=\sum_{j=1}^mx_je_j$. An $\R_{0,m}$-valued function $f\in C^1(\Omega)$, is say to be left $\varphi$-hyperholomorphic (respectively, right) in $\Omega$ if $\hD f = 0$ ($f\hD = 0$) in $\Omega$. It is easy to check that the function $K_\varphi$ is two-sided $\varphi$-hyperholomorphic. We define the inner and outer products between $\hD$ and a $k$-vector field $F_k$ by $$\hD\bullet F_k:=\frac{1}{2}[\hD F_k-(-1)^kF_k\hD]=[\hD F_k]_{k-1}$$ and $$\hD\wedge F_k:=\frac{1}{2}[\hD F_k+(-1)^kF_k\hD]=[\hD F_k]_{k+1}.$$

The Stokes theorem in Clifford analysis yields the Borel–Pompeiu
integral representation formula for $\R_{0,m}$-valued functions $f\in C^1(\overline{\Omega})$. Actually,
\begin{equation}\label{BPcla}
f(\ux)=(\mathcal{C}_\varphi^l f)(\ux)+(\mathcal{T}_\varphi^l \hD f)(\ux),\quad\ux\in\Omega,
\end{equation}
where
\begin{equation}
(\mathcal{C}_\varphi^l f)(\ux):=\int_{\partial\Omega}K_\varphi(\uy-\ux)n_\varphi(\uy)f(\uy)d\uy,\quad\ux\notin\partial\Omega,
\end{equation}
and
\begin{equation}
(\mathcal{T}_\varphi^l f)(\ux):=-\int_{\Omega}K_\varphi(\uy-\ux)f(\uy)d\uy,
\end{equation}
are the Cauchy and Teodorescu transforms, respectively. Here, $n_\varphi(\uy)=\sum_{j=1}^mn_j(\uy)\varphi_j$, where $n_j(\uy)$ is the $j$-th component of the outward unit normal vector on $\partial\Omega$ at the point $\uy\in\partial\Omega$. The right-handed version of equation \eqref{BPcla} reads  
\[
f(\ux)=(\mathcal{C}_\varphi^r f)(\ux)+(\mathcal{T}_\varphi^r f\hD)(\ux),\quad\ux\in\Omega,
\]
where
\[
(\mathcal{C}_\varphi^r f)(\ux):=\int_{\partial\Omega}f(\uy)n_\varphi(\uy)K_\varphi(\uy-\ux)d\uy,\quad\ux\notin\partial\Omega,
\]
and
\[
(\mathcal{T}_\varphi^r f)(\ux):=-\int_{\Omega}f(\uy)K_\varphi(\uy-\ux)d\uy.
\]

The fundamental solution of the operator $\hD\pD(.)$ is given by
\begin{equation}\label{kernel}
K_{\varphi,\psi}(\ux)=\pD\hD\left[\frac{\|\ux\|^{4-m}}{\sigma_m(2-m)(4-m)}\right]=\frac{(2-m)\|\ux\|^{-m}\ux_\psi\ux_{\varphi}+\|\ux\|^{2-m}\sum_{j=1}^m\psi_j\varphi_j}{2\sigma_m(2-m)}.
\end{equation}
It is easy to see that $K_{\varphi,\psi}$ satisfies the following important relation
\begin{equation}
\pD K_{\varphi,\psi}(\ux)=K_\varphi(\ux),
\end{equation}
which implies that the Cauchy and Teodorescu transforms associated to $\hD\pD(.)$ may be defined by
\begin{equation}\label{Thar}
(\mathcal{C}_{\varphi,\psi}f)(\ux):=-\int_{\partial\Omega}K_{\varphi,\psi}(\uy-\ux)n_\varphi(\uy)f(\uy)d\uy
\end{equation}
and
\begin{equation}\label{TThar}
(\mathcal{T}_{\varphi,\psi}f)(\ux):=\int_{\Omega}K_{\varphi,\psi}(\uy-\ux)f(\uy)d\uy,
\end{equation}
respectively. 

Notice that $\pD(\mathcal{C}_{\varphi,\psi}f)=\mathcal{C}_{\varphi}^lf$ and, therefore, $\hD\pD(\mathcal{C}_{\varphi,\psi}f)=0$ in $\R^m\setminus\partial\Omega$. Similarly,  $\pD(\mathcal{T}_{\varphi,\psi}f)=\mathcal{T}_{\varphi}^lf$ and, thus, $\hD\pD(\mathcal{T}_{\varphi,\psi}f)=f$ in $\Omega$. 

Taking into account \eqref{kernel}, the integral transforms \eqref{Thar} and \eqref{TThar} admit the rewritten 
\begin{align}
(\mathcal{C}_{\varphi,\psi}f)(\ux)&=\frac{1}{2}\left(\int_{\partial\Omega}(\uy_\psi-\ux_\psi)K_\varphi(\uy-\ux)n_\varphi(\uy)f(\uy)d\uy+\sum_{j=1}^m\int_{\partial\Omega}\psi_j\varphi_jE_1(\uy-\ux)n_\varphi(\uy)f(\uy)d\uy\right),\label{Tharmr}\\
(\mathcal{T}_{\varphi,\psi}f)(\ux)&=-\frac{1}{2}\left(\int_{\Omega}(\uy_\psi-\ux_\psi)K_\varphi(\uy-\ux)f(\uy)d\uy+\sum_{j=1}^m\int_{\Omega}\psi_j\varphi_jE_1(\uy-\ux)f(\uy)d\uy\right).\label{TTharmr}
\end{align}

The Cauchy and Teodorescu transforms associated to $\hD(.)\pD$, which have been introduced in \cite{DAS1}. 

These integral operators have the form
\begin{equation}\label{tinfra}
(\mathcal{C}_{\varphi,\psi}^\textsf{infra}f)(\ux):=\frac{1}{2}\left(\int_{\partial\Omega}K_\varphi(\uy-\ux)n_\varphi(\uy)f(\uy)(\uy_\psi-\ux_\psi)d\uy+\sum_{j=1}^m\varphi_j\int_{\partial\Omega} E_1(\uy-\ux)n_\varphi(\uy)f(\uy)d\uy\psi_j\right),
\end{equation}
and
\begin{equation}\label{ttinfra}
(\mathcal{T}_{\varphi,\psi}^\textsf{infra}f)(\ux):=-\frac{1}{2}\left(\int_{\Omega}K_\varphi(\uy-\ux)f(\uy)(\uy_\psi-\ux_\psi)d\uy+\sum_{j=1}^m\varphi_j\int_\Omega E_1(\uy-\ux)f(\uy)d\uy\psi_j\right),
\end{equation}
and satisfy the relations $(\mathcal{C}_{\varphi,\psi}^\textsf{infra}f)\pD=\mathcal{C}_\varphi^l f$ and $(\mathcal{T}_{\varphi,\psi}^\textsf{infra}f)\pD=\mathcal{T}_\varphi^l f$, in $\Omega$, which leads to $\hD(\mathcal{C}_{\varphi,\psi}^\textsf{infra}f)\pD=0$ in $\R^m\setminus\partial\Omega$ and  
$\hD(\mathcal{T}_{\varphi,\psi}^\textsf{infra}f)\pD=f$ in $\Omega$, respectively.

The null solutions of operators $\hD(.)\pD$ and $\hD\pD(.)$ are known as $(\varphi,\psi)$-inframonogenic and $(\varphi,\psi)$-harmonic functions, respectively \cite{DAS1,DAS2,bory,serrano}.

Let us now prove the following formula, which will be used in the sequel:
\begin{equation}
\pD\left[(\mathcal{T}_{\varphi,\psi}^\textsf{infra}f)(\ux)\right]=\left[(\mathcal{T}_{\varphi,\psi}f)(\ux)\right]\pD,\quad\ux\in\Omega.
\end{equation}

Indeed, we have
\begin{align*}
&-2\pD\left[(\mathcal{T}_{\varphi,\psi}^\textsf{infra}f)(\ux)\right]\\&=-\frac{1}{\sigma_m}\sum_{j=1}^m\int_\Omega \psi_j\left( -\varphi_j\|\uy-\ux\|^{-m}+m(\uy_\varphi-\ux_\varphi)\|\uy-\ux\|^{-m-1}\frac{y_j-x_j}{\|\uy-\ux\|}\right)f(\uy)(\uy_\psi-\ux_\psi)d\uy\\
&\quad-\sum_{j=1}^m\int_\Omega\psi_jK_{\varphi}(\uy-\ux)f(\uy)\psi_jd\uy-\sum_{j=1}^m\int_\Omega K_{\psi}(\uy-\ux)\varphi_j f(\uy)\psi_jd\uy\\
&=-\int_\Omega\sum_{j=1}^m\psi_j\varphi_jf(\uy)K_\psi(\uy-\ux)d\uy-\frac{m}{\sigma_m}\int_\Omega(\uy_\psi-\ux_\psi)(\uy_\varphi-\ux_\varphi)f(\uy)(\uy_\psi-\ux_\psi)\|\uy-\ux\|^{-m-2}d\uy\\
&\quad-\int_\Omega\sum_{j=1}^m[\psi_jK_\varphi(\uy-\ux)+K_\psi(\uy-\ux)\varphi_j]f(\uy)\psi_jd\uy,
\end{align*}
\begin{align*}
&-2\left[(\mathcal{T}_{\varphi,\psi}f)(\ux)\right]\pD\\&=-\frac{1}{\sigma_m}\sum_{j=1}^m\int_\Omega \left(m\|\uy-\ux\|^{-m-2}(y_j-x_j)(\uy_\psi-\ux_\psi)(\uy_\varphi-\ux_\varphi)-\|\uy-\ux\|^{-m}\psi_j(\uy_\varphi-\ux_\varphi)-\|\uy-\ux\|^{-m}(\uy_\psi-\ux_\psi)\varphi_j\right) f(\uy)\psi_j\\
&\quad+\frac{1}{\sigma_m}\sum_{j=1}^m\psi_j\varphi_j\int_\Omega f(\uy)\|\uy-\ux\|^{-m}(\uy_\psi-\ux_\psi)d\uy\\
&=-\int_\Omega\sum_{j=1}^m\psi_j\varphi_jf(\uy)K_\psi(\uy-\ux)d\uy-\frac{m}{\sigma_m}\int_\Omega(\uy_\psi-\ux_\psi)(\uy_\varphi-\ux_\varphi)f(\uy)(\uy_\psi-\ux_\psi)\|\uy-\ux\|^{-m-2}d\uy\\
&\quad-\int_\Omega\sum_{j=1}^m[\psi_jK_{\varphi}(\uy-\ux)+K_\psi(\uy-\ux)\varphi_j]f(\uy)\psi_jd\uy.
\end{align*}
Similarly, we also obtain that
\begin{equation}
\pD\left[(\mathcal{C}_{\varphi,\psi}^\textsf{infra}f)(\ux)\right]=\left[(\mathcal{C}_{\varphi,\psi}f)(\ux)\right]\pD,\quad\ux\in\R^m\setminus\partial\Omega.
\end{equation}

The Cauchy and Teodorescu transforms associated to $\La$ will be defined as a linear combination of the corresponding integral operators \eqref{Tharmr}, \eqref{TTharmr}, \eqref{tinfra} and \eqref{ttinfra}, a linear combination which involves the Lam\'e parameters. Indeed, we define 
\begin{equation}
(\mathcal{C}_{\varphi,\psi}^\dagger f)(\ux):=-\frac{\mu+\lambda}{2\mu(2\mu+\lambda)}(\mathcal{C}_{\varphi,\psi}^\textsf{infra}f)(\ux)+\frac{3\mu+\lambda}{2\mu (2\mu+\lambda)}(\mathcal{C}_{\varphi,\psi}f)(\ux),
\end{equation}
and
\begin{equation}
(\mathcal{T}_{\varphi,\psi}^\dagger f)(\ux):=-\frac{\mu+\lambda}{2\mu (2\mu+\lambda)}(\mathcal{T}_{\varphi,\psi}^\textsf{infra}f)(\ux)+\frac{3\mu+\lambda}{2\mu (2\mu+\lambda)}(\mathcal{T}_{\varphi,\psi}f)(\ux).
\end{equation}
\begin{proposition}
Let $f\in C^2(\Omega)$. Then $\mathcal{T}_{\varphi,\psi}^\dagger f$ is also twice differentiable in $\Omega$ and
\begin{equation}
\La(\mathcal{T}_{\varphi,\psi}^\dagger f(\ux))=f(\ux),\quad\ux\in\Omega.
\end{equation}
\end{proposition}
\begin{proof}
We have
\begin{align*}
&\left(\frac{\mu+\lambda}{2}\right)(\mathcal{T}_{\varphi,\psi}^\dagger f(\ux))\pD+\left(\frac{3\mu+\lambda}{2}\right)\pD (\mathcal{T}_{\varphi,\psi}^\dagger f(\ux))\\&=-\frac{(\mu+\lambda)^2}{4\mu(2\mu+\lambda)}(\mathcal{T}_\varphi^l f)(\ux)+\frac{(3\mu+\lambda)^2}{4\mu(2\mu+\lambda)}(\mathcal{T}_\varphi^l f)(\ux)\\
&\quad-\frac{(\mu+\lambda)(3\mu+\lambda)}{4\mu(2\mu+\lambda)}\pD\left[(\mathcal{T}_{\varphi,\psi}^\textsf{infra}f)(\ux)\right]+\frac{(\mu+\lambda)(3\mu+\lambda)}{4\mu(2\mu+\lambda)}\left[(\mathcal{T}_{\varphi,\psi}f)(\ux)\right]\pD\\
&=(\mathcal{T}_\varphi^l f)(\ux).
\end{align*}
Thus,
$$\La(\mathcal{T}_{\varphi,\psi}^\dagger f(\ux))=\hD[(\mathcal{T}_\varphi^l f)(\ux)]=f(\ux),\quad\ux\in\Omega.
$$
\end{proof}
Analogously, we can prove that $\La(\mathcal{C}_{\varphi,\psi}^\dagger f(\ux))=0$ in $\R^m\setminus\partial\Omega$.

\begin{theorem}[Borel-Pompeiu formula] Let $f\in C^2(\overline{\Omega})$. Then
\begin{align}
&\frac{(3\mu+\lambda)^2}{4\mu (2\mu+\lambda)}(\mathcal{C}_\psi^lf)(\ux)-\frac{(\mu+\lambda)^2}{4\mu (2\mu+\lambda)}(\mathcal{C}_\psi^rf)(\ux)+\mathcal{C}_{\varphi,\psi}^\dagger\left(\frac{\mu+\lambda}{2}f\pD+\frac{3\mu+\lambda}{2}\pD f\right)(\ux)+\mathcal{T}_{\varphi,\psi}^\dagger(\La f)(\ux)\nonumber\\
&+\frac{(3\mu+\lambda)(\mu+\lambda)}{4\mu (2\mu+\lambda)}\left(\int_{\partial\Omega}K_\psi(\uy-\ux)f(\uy)n_\psi(\uy)d\uy-\int_{\partial\Omega}n_\psi(\uy)f(\uy)K_\psi(\uy-\ux)d\uy\right)\nonumber\\
&=\left\{\begin{array}{rl}
f(\ux),&\text{if}\;\ux\in\Omega,\\
0,&\text{if}\;\ux\in\R^m\setminus\overline{\Omega}.
\end{array}\right.\label{bpform}
\end{align}
\end{theorem}
\begin{proof}
We prove only the statement for $\ux\in\Omega$. The proof for $\ux\in\R^m\setminus\overline{\Omega}$ is completely analogous.

Using the Borel-Pompeiu formulas involving the operators $\hD\pD(.)$ and $\hD(.)\pD$ (see \cite{DAS1} and \cite{bory}):
\begin{equation}\label{fbph}
f(\ux)=(\mathcal{C}_\psi^l f)(\ux)+(\mathcal{C}_{\varphi,\psi}\pD f)(\ux)+(\mathcal{T}_{\varphi,\psi}\hD\pD f)(\ux) 
\end{equation}
and
\begin{equation}\label{fbpi}
f(\ux)=(\mathcal{C}_\psi^rf)(\ux)+(\mathcal{C}_{\varphi,\psi}^\textsf{infra}f\pD)(\ux)+(\mathcal{T}_{\varphi,\psi}^\textsf{infra}\hD f\pD)(\ux),
\end{equation}
we have
\begin{align*}
\mathcal{T}_{\varphi,\psi}^\dagger(\La f)(\ux)&=-\frac{(\mu+\lambda)^2}{4\mu(2\mu+\lambda)}(\mathcal{T}_{\varphi,\psi}^\textsf{infra}\hD f\pD)(\ux)-\frac{(\mu+\lambda)(3\mu+\lambda)}{4\mu(2\mu+\lambda)}(\mathcal{T}_{\varphi,\psi}^\textsf{infra}\hD\pD f)(\ux)\\
&\quad+\frac{(3\mu+\lambda)(\mu+\lambda)}{4\mu(2\mu+\lambda)}(\mathcal{T}_{\varphi,\psi}\hD f\pD)(\ux)+\frac{(3\mu+\lambda)^2}{4\mu(2\mu+\lambda)}(\mathcal{T}_{\varphi,\psi}\hD\pD f)(\ux)\\
&=f(\ux)+\frac{(\mu+\lambda)^2}{4\mu(2\mu+\lambda)}(\mathcal{C}_\psi^rf)(\ux)-\frac{(3\mu+\lambda)^2}{4\mu(2\mu+\lambda)}(\mathcal{C}_\psi^lf)(\ux)\\
&\quad+ \frac{(\mu+\lambda)^2}{4\mu(2\mu+\lambda)}(\mathcal{C}_{\varphi,\psi}^\textsf{infra} f\pD)(\ux)-\frac{(3\mu+\lambda)^2}{4\mu(2\mu+\lambda)}(\mathcal{C}_{\varphi,\psi}\pD f)(\ux)\\
&\quad+\frac{(3\mu+\lambda)(\mu+\lambda)}{4\mu(2\mu+\lambda)}[(\mathcal{T}_{\varphi,\psi}\hD f\pD)(\ux)-(\mathcal{T}_{\varphi,\psi}^\textsf{infra}\hD\pD f)(\ux)]. 
\end{align*}
Now, applying the equations \eqref{fbph} and \eqref{fbpi}, and using the identity
\begin{equation}
(\mathcal{T}_\psi^l f)\pD=\pD(\mathcal{T}_\psi^r f),
\end{equation}
we get
\begin{equation}
(\mathcal{T}_{\varphi,\psi}\hD f\pD)(\ux)-(\mathcal{T}_{\varphi,\psi}^\textsf{infra}\hD\pD f)(\ux)=(\mathcal{T}_\psi^l f\pD)(\ux)-(\mathcal{T}_\psi^r \pD f)(\ux)+(\mathcal{C}_{\varphi,\psi}^\textsf{infra}\pD f)(\ux)-(\mathcal{C}_{\varphi,\psi} f\pD)(\ux). 
\end{equation}

We obtain that
\begin{align*}
&(\mathcal{T}_\psi^l f\pD)(\ux)-(\mathcal{T}_\psi^r \pD f)(\ux)\\&=\sum_{k=0}^m\int_\Omega (\pDy [f(\uy)]_k)K_\psi(\uy-\ux)d\uy-\sum_{k=0}^m\int_\Omega K_\psi(\uy-\ux)([f(\uy)]_k\pDy)d\uy\\
&=2\sum_{k=0}^m\int_\Omega (\pDy\bullet [f(\uy)]_k)\bullet K_\psi(\uy-\ux)d\uy+2\sum_{k=0}^{m}\int_\Omega (\pDy\wedge [f(\uy)]_k)\wedge K_\psi(\uy-\ux)d\uy\\
&=2\sum_{k=0}^m(-1)^{k+1}\int_\Omega ( [f(\uy)]_k\bullet \pDy)\bullet K_\psi(\uy-\ux)d\uy+2\sum_{k=0}^{m}(-1)^k\int_\Omega ([f(\uy)]_k\wedge\pDy)\wedge K_\psi(\uy-\ux)d\uy\\
&=2\sum_{k=0}^m(-1)^{k+1}\int_{\partial\Omega} ([f(\uy)]_k\bullet n_\psi(\uy))\bullet K_\psi(\uy-\ux)d\uy+2\sum_{k=0}^{m}(-1)^{k}\int_{\partial\Omega} ([f(\uy)]_k\wedge n_\psi(\uy))\wedge K_\psi(\uy-\ux)d\uy\\
&=\frac{1}{2}\sum_{k=0}^m\int_{\partial\Omega}[(-1)^{k+1}([f(\uy)]_kn_\psi(\uy)K_{\psi}(\uy-\ux)-(-1)^kn_\psi(\uy)[f(\uy)]_kK_\psi(\uy-\ux)))\\&\quad-K_{\psi}(\uy-\ux)[f(\uy)]_kn_\psi(\uy)+(-1)^kK_\psi(\uy-\ux)n_\psi(\uy)[f(\uy)]_k)]d\uy\\
&\quad+\frac{1}{2}\sum_{k=0}^m\int_{\partial\Omega}[(-1)^{k}([f(\uy)]_kn_\psi(\uy)K_{\psi}(\uy-\ux)+(-1)^kn_\psi(\uy)[f(\uy)]_kK_\psi(\uy-\ux)))\\&\quad-K_{\psi}(\uy-\ux)[f(\uy)]_kn_\psi(\uy)-(-1)^kK_\psi(\uy-\ux)n_\psi(\uy)[f(\uy)]_k)]d\uy\\
&=\int_{\partial\Omega}n_\psi(\uy)f(\uy)K_\psi(\uy-\ux)d\uy-\int_{\partial\Omega}K_\psi(\uy-\ux)f(\uy)n_\psi(\uy)d\uy.
\end{align*}

Therefore,
\begin{align*}
&f(\ux)\\&=\frac{(3\mu+\lambda)^2}{4\mu(2\mu+\lambda)}(\mathcal{C}_\psi^lf)(\ux)-\frac{(\mu+\lambda)^2}{4\mu(2\mu+\lambda)}(\mathcal{C}_\psi^rf)(\ux)-\frac{(\mu+\lambda)^2}{4\mu(2\mu+\lambda)}(\mathcal{C}_{\varphi,\psi}^\textsf{infra}f\pD)(\ux)+\frac{(3\mu+\lambda)^2}{4\mu(2\mu+\lambda)}(\mathcal{C}_{\varphi,\psi}\pD f)(\ux)\\
&\quad+\frac{(3\mu+\lambda)(\mu+\lambda)}{4\mu(2\mu+\lambda)}(\mathcal{C}_{\varphi,\psi}f\pD)(\ux)-\frac{(3\mu+\lambda)(\mu+\lambda)}{4\mu(2\mu+\lambda)}(\mathcal{C}_{\varphi,\psi}^\textsf{infra}\pD f)(\ux)+\mathcal{T}_{\varphi,\psi}^\dagger(\La f)(\ux)\\
&\quad+\frac{(3\mu+\lambda)(\mu+\lambda)}{4\mu(2\mu+\lambda)}\left(\int_{\partial\Omega}K_\psi(\uy-\ux)f(\uy)n_\psi(\uy)d\uy-\int_{\partial\Omega}n_\psi(\uy)f(\uy)K_\psi(\uy-\ux)d\uy\right)\\
&=\frac{(3\mu+\lambda)^2}{4\mu(2\mu+\lambda)}(\mathcal{C}_\psi^lf)(\ux)-\frac{(\mu+\lambda)^2}{4\mu(2\mu+\lambda)}(\mathcal{C}_\psi^rf)(\ux)+\mathcal{C}_{\varphi,\psi}^\dagger\left(\frac{\mu+\lambda}{2}f\pD+\frac{3\mu+\lambda}{2}\pD f\right)(\ux)+\mathcal{T}_{\varphi,\psi}^\dagger(\La f)(\ux)\\
&\quad+\frac{(3\mu+\lambda)(\mu+\lambda)}{4\mu(2\mu+\lambda)}\left(\int_{\partial\Omega}K_\psi(\uy-\ux)f(\uy)n_\psi(\uy)d\uy-\int_{\partial\Omega}n_\psi(\uy)f(\uy)K_\psi(\uy-\ux)d\uy\right),
\end{align*}	
and we are done.
\end{proof}
Looking at the above proof it is easy to check that the Borel-Pompeiu
formula \eqref{bpform}  will remain valid if the condition $f\in C^2(\overline{\Omega})$ is replaced by the
weaker assumption
\begin{equation}
f\in C^2(\Omega)\cap C^1(\overline{\Omega}),\quad\int_\Omega\|\Lay f(\uy)\|d\uy<+\infty.
\end{equation}
As we mentioned before, the above theorem yields a Cauchy representation formula.
\begin{corollary}[Cauchy integral formula]
Let $f\in C^2(\Omega)\cap C^1(\overline{\Omega})$. If $\La f=0$ in $\Omega$ then
\begin{align}
f(\ux)&=\frac{(3\mu+\lambda)^2}{4\mu (2\mu+\lambda)}(\mathcal{C}_\psi^lf)(\ux)-\frac{(\mu+\lambda)^2}{4\mu (2\mu+\lambda)}(\mathcal{C}_\psi^rf)(\ux)+\mathcal{C}_{\varphi,\psi}^\dagger\left(\frac{\mu+\lambda}{2}f\pD+\frac{3\mu+\lambda}{2}\pD f\right)(\ux)\nonumber\\
&\quad+\frac{(3\mu+\lambda)(\mu+\lambda)}{4\mu (2\mu+\lambda)}\left(\int_{\partial\Omega}K_\psi(\uy-\ux)f(\uy)n_\psi(\uy)d\uy-\int_{\partial\Omega}n_\psi(\uy)f(\uy)K_\psi(\uy-\ux)d\uy\right).\label{Cauchyf}
\end{align}
\end{corollary}

\section{Jump problem}\label{section3}
   
The Cauchy formula \eqref{Cauchyf} says that if there exists a solution of the Dirichlet problem 
\begin{equation}\label{DP}
\Biggl\{
\begin{array}{rl}
\La F=0\,\,\mbox{in}\,\,\Omega,\\
F=f\,\,\mbox{on}\,\,\partial\Omega,
\end{array} 
\end{equation}
in the class $C^2(\Omega)\cap C^1(\overline{\Omega})$, then a solution of \eqref{DP} can be represented by
\begin{align}
F(\ux)&=\frac{(3\mu+\lambda)^2}{4\mu (2\mu+\lambda)}(\mathcal{C}_\psi^lf)(\ux)-\frac{(\mu+\lambda)^2}{4\mu (2\mu+\lambda)}(\mathcal{C}_\psi^rf)(\ux)+\mathcal{C}_{\varphi,\psi}^\dagger\left(\frac{\mu+\lambda}{2}f\pD+\frac{3\mu+\lambda}{2}\pD f\right)(\ux)\nonumber\\
&\quad+\frac{(3\mu+\lambda)(\mu+\lambda)}{4\mu (2\mu+\lambda)}\left(\int_{\partial\Omega}K_\psi(\uy-\ux)f(\uy)n_\psi(\uy)d\uy-\int_{\partial\Omega}n_\psi(\uy)f(\uy)K_\psi(\uy-\ux)d\uy\right).\label{Cauchyf1}
\end{align}

However, because of the absence of a maximum principle for null solutions of generalized Lam\'e-Navier systems, the usual uniqueness proof in the case of the Dirichlet problem \eqref{DP} loses its validity and other solutions in the class $C^2(\Omega)\cap C(\overline{\Omega})$ may exist.  See the following example, which is essentially the same as the one presented in \cite{DAS6}, but we have decided to include it here for the sake of completeness.

\begin{example}
The vector field $F(\ux)=(6x_1^2+x_2^2+x_3^2-1)e_2$ in $\R^3$ is a non-zero solution of the homogeneous Dirichlet problem 
\begin{equation}\label{DPIcont}
\Biggl\{
\begin{array}{rl}
0.2\cdot\hD F(\ux)\pux+0.3\cdot\hD\pux F(\ux)=0 &\,\,\mbox{\text{in}}\,\,\Omega,\\F(\ux)=0&\,\,\mbox{\text{on}}\,\,\partial\Omega,
\end{array} 
\end{equation}
where $\varphi=\{-e_1,e_2,e_3\}$ and $\Omega=\{\ux\in \R^3: 6x_1^2+x_2^2+x_3^2<1\}$.

\end{example}
The above counterexample inspires the consideration of different boundary value problems, where such a difficulty may be overcome.

For simplicity, let us introduce the first-order differential operator $\mathcal{M}:=\left(\frac{\mu+\lambda}{2}\right)(.)\pD+\frac{3\mu+\lambda}{2}\pD(.)$ and consider the following jump problem for the Lam\'e-Navier system \eqref{gLNs}:
\begin{equation}\label{JP2}
\begin{cases}
\La F=0,\,&\ux\in\Omega_+\cup\Omega_-,\\
{[F]}^+(\ux)-{[F]}^-(\ux)=f(\ux),\,&\ux\in\partial\Omega,\\
[\mathcal{M}F]^+-[\mathcal{M}F]^-=f_1(\ux),\,&\ux\in\partial\Omega,\\
F(\infty)=(\mathcal{M}F)(\infty)=0,
\end{cases}
\end{equation}
where $f,f_1$ are assumed in $C^{0,\nu}(\partial\Omega)$. 
\begin{theorem}
Let $\partial\Omega$ be sufficiently smooth. Then the jump problem \eqref{JP2} has a unique solution given by
\begin{align}
F(\ux)&=\frac{(3\mu+\lambda)^2}{4\mu (2\mu+\lambda)}(\mathcal{C}_\psi^lf)(\ux)-\frac{(\mu+\lambda)^2}{4\mu (2\mu+\lambda)}(\mathcal{C}_\psi^rf)(\ux)+(\mathcal{C}_{\varphi,\psi}^\dagger f_1)(\ux)\nonumber\\
&\quad+\frac{(3\mu+\lambda)(\mu+\lambda)}{4\mu (2\mu+\lambda)}\left(\int_{\partial\Omega}K_\psi(\uy-\ux)f(\uy)n_\psi(\uy)d\uy-\int_{\partial\Omega}n_\psi(\uy)f(\uy)K_\psi(\uy-\ux)d\uy\right).\label{JPf1}
\end{align}
\end{theorem}
\begin{proof}
We obtain
\begin{align*}
(\mathcal{M}F)(\ux)&=\frac{(\mu+\lambda)(3\mu+\lambda)^2}{8\mu (2\mu+\lambda)}[(\mathcal{C}_\psi^lf)(\ux)]\pD-\frac{(3\mu+\lambda)(\mu+\lambda)^2}{8\mu (2\mu+\lambda)}\pD[(\mathcal{C}_\psi^rf)(\ux)]+(\mathcal{C}_\varphi^lf_1)(\ux)\\
&\quad+\frac{(3\mu+\lambda)(\mu+\lambda)^2}{8\mu (2\mu+\lambda)}\left[\left(\int_{\partial\Omega}K_\psi(\uy-\ux)f(\uy)n_\psi(\uy)d\uy\right)\pD\right]\\
&\quad+\frac{(3\mu+\lambda)^2(\mu+\lambda)}{8\mu (2\mu+\lambda)}\left[-\pD\left(\int_{\partial\Omega}n_\psi(\uy)f(\uy)K_\psi(\uy-\ux)d\uy\right)\right]\\
&=(C_\varphi^lf_1)(\ux),
\end{align*}
whence $\La F(\ux)=\hD[\mathcal{M}F](\ux)=0$ for $\ux\in\Omega_+\cup\Omega_-$. In this last calculation, we have used that
\begin{align*}
&[(\mathcal{C}_\psi^lf)(\ux)]\pD-\pD\left(\int_{\partial\Omega}n_\psi(\uy)f(\uy)K_\psi(\uy-\ux)d\uy\right)\\&=\sum_{j=1}^m\int_{\partial\Omega}\frac{1}{\sigma_m}\left[-m\|\uy-\ux\|^{-m-2}(\uy_\psi-\ux_\psi)(y_j-x_j)+\psi_j\|\uy-\ux\|^{-m}\right]n_\psi(\uy)f(\uy)\psi_jd\uy\\
&\quad-\sum_{j=1}^m\int_{\partial\Omega}\frac{1}{\sigma_m}\psi_jn_\psi(\uy)f(\uy)\left[-m\|\uy-\ux\|^{-m-2}(\uy_\psi-\ux_\psi)(y_j-x_j)+\psi_j\|\uy-\ux\|^{-m}\right]d\uy\\
&=0,
\end{align*}
\begin{align*}
&-\pD[(\mathcal{C}_\psi^rf)(\ux)]+\left(\int_{\partial\Omega}K_\psi(\uy-\ux)f(\uy)n_\psi(\uy)d\uy\right)\pD\\
&=-\sum_{j=1}^m\int_{\partial\Omega}\frac{1}{\sigma_m}\psi_jf(\uy)n_\psi(\uy)\left[-m\|\uy-\ux\|^{-m-2}(\uy_\psi-\ux_\psi)(y_j-x_j)+\psi_j\|\uy-\ux\|^{-m}\right]d\uy\\
&\quad+\sum_{j=1}^m\int_{\partial\Omega}\frac{1}{\sigma_m}\left[-m\|\uy-\ux\|^{-m-2}(\uy_\psi-\ux_\psi)(y_j-x_j)+\psi_j\|\uy-\ux\|^{-m}\right]f(\uy)n_\psi(\uy)\psi_jd\uy\\
&=0.
\end{align*}

Now we will check the boundary conditions. The integrals of transform $\mathcal{C}_{\varphi,\psi}^\dagger$ become weakly singular as $\ux$ approaches to $\partial\Omega$, which is easy to check. Consequently, all these
integrals do not have jump discontinuity in $\partial\Omega$. Combining this fact with the classical Plemelj–Sokhotski formulas applied to the Cliffordian–Cauchy-type integrals in \eqref{JPf1}, we conclude that 
\begin{align*}
[F]^+(\ux)-[F]^-(\ux)=\frac{(3\mu+\lambda)^2}{4\mu (2\mu+\lambda)}f(\ux)-\frac{(\mu+\lambda)^2}{4\mu (2\mu+\lambda)}f(\ux)=f(\ux),\quad\ux\in\partial\Omega.
\end{align*}
On the other hand, we have
\begin{align*}
[\mathcal{M}F]^+-[\mathcal{M}F]^-&=\left(\frac{\mu+\lambda}{2}\right)({[F\pD]}^+(\ux)-{[F\pD]}^-(\ux))+\left(\frac{3\mu+\lambda}{2}\right)({[\pD F]}^+(\ux)-{[\pD F]}^-(\ux))\\&=[C_\varphi^lf_1]^+-[C_\varphi^lf_1]^-\\&=f_1(\ux),\quad\ux\in\partial\Omega.
\end{align*}
It is a matter of routine to check the vanishing conditions $F(\infty) =(\mathcal{M}F)(\infty)=0$.

To prove the uniqueness of the solution, we are reduced to prove that the problem
\begin{equation}\label{JP2hom}
\begin{cases}
\La F=0,\,&\ux\in\Omega_+\cup\Omega_-,\\
{[F]}^+(\ux)-{[F]}^-(\ux)=0,\,&\ux\in\partial\Omega,\\
[\mathcal{M}F]^+(\ux)-[\mathcal{M}F]^-(\ux)=0,\,&\ux\in\partial\Omega,\\
F(\infty)=(\mathcal{M}F)(\infty)=0,
\end{cases}
\end{equation}
allows only the null solution $F\equiv 0$.  By a combination
of Painleve and Liouville theorems in Clifford analysis (see \cite{BDS}), we obtain that $(\mathcal{M}F)(\ux)=0$ in $\R^m$. Thanks to the decay condition, we can say that $F$ is bounded in $\R^m$, that is, $\|F(\ux)\|\leq R$ for all $\ux\in\R^m$. Let $\ux_0$ be an arbitrary point in $\R^m$ and let $\rho>0$. Using the Cauchy integral representation formula, we have that in $B_\rho(\ux_0)=\{\ux\in\R^m: \|\ux-\ux_0\|<\rho\}$
\begin{align*}
F(\ux)&=\frac{(3\mu+\lambda)^2}{4\mu (2\mu+\lambda)}\int_{\partial B_\rho(\ux_0)}K_\psi(\uy-\ux)n_\psi(\uy)F(\uy)d\uy-\frac{(\mu+\lambda)^2}{4\mu (2\mu+\lambda)}\int_{\partial B_\rho(\ux_0)}F(\uy)n_\psi(\uy)K_\psi(\uy-\ux)d\uy\\&\quad+\frac{(3\mu+\lambda)(\mu+\lambda)}{4\mu (2\mu+\lambda)}\left(\int_{\partial B_\rho(\ux_0)}K_\psi(\uy-\ux)F(\uy)n_\psi(\uy)d\uy-\int_{\partial B_\rho(\ux_0)}n_\psi(\uy)F(\uy)K_\psi(\uy-\ux)d\uy\right). 
\end{align*}
The Clifford-Cauchy kernel $K_\psi(\uy-\ux)$ is arbitrarily often continuously differentiable with
respect to the $x_j$, so one can differentiate under the integral sign and get the integral for
the derivatives of $F$. We then use the formula of \cite[Corollary 7.28, p.143]{GHS} for the first derivative of $F$
$$\|\partial_{x_j}F(\ux_0)\|\leq\frac{RC}{\rho},$$
where $C$ is a constant that depends on the dimension $m$. As $\rho$ is here an arbitrary real number, for $\rho\to\infty$ we can conclude that $\partial_{x_j}F(\ux_0)=0$. Because of the arbitrary choice of $\ux_0$ all derivatives $\partial_{x_j}F$ are zero at all points in $\R^m$ and $F$ is constant. Since $F(\infty)=0$, then necessarily $F\equiv 0$, and this completes the proof.  
\end{proof}

When one omits the smoothness conditions on $\partial\Omega$, the function \eqref{JPf1} does not represent in general a solution of \eqref{JP2} or, which is even more inconvenient, it loses its usual sense. The natural question arises whether it is possible to construct a solution of \eqref{JP2}, analogous to \eqref{JPf1}, where this time $\Omega$ is assumed to be a domain with fractal boundary $\partial\Omega$. Fractals are used to model complex phenomena and repetitive patterns found in nature and in various disciplines, such as physics, biology, engineering, and art. They are used to understand, analyze, and represent the complexity of structures, as well as to compress data, create visual effects, and develop new technologies (see, for example \cite{Be,Kar,Tu}). It is for these reasons that it is not unreasonable to consider the above problem  under such a general geometric conditions.

In the case of fractal domains, the merely H\"older continuity of the boundary traces $f,f_1$ does not, in general, offer the possibility of constructing the solution of \eqref{JP2}. The method we will mainly sketch below has its roots in the seminal work of B. Kats \cite{Kats} (see also \cite{AAB}).

Firstly, we define an appropriate way of measuring the fractality of $\partial\Omega$. In this direction we prefer the concept of $d$-summable boundary, introduced by Harrison and Norton in \cite{HN}. We say that $\partial\Omega$ is $d$-summable for some $m-1<d<m$, if the improper integral
\[
\int_0^1 N_{\partial\Omega}(\tau) \, \tau^{d-1} \, d\tau
\]
converges, where $ N_{\partial\Omega}(\tau)$ stands for the minimal number of balls of radius $\tau$ needed to cover $\partial\Omega$. 

Secondly, we assume that $f$ belongs (component-wisely) to the so-called higher order Lipschitz class $\Li(1+\nu,\partial\Omega)$ introduced by Whitney \cite{whitney} and more extensively studied in \cite{stein}. Roughly speaking, this means that there exist a jet $\{f,f^{\bj},\,|\bj|=1\}$ 
such that the field of polynomials (in the usual multi-index notation) 
\[f(\uy)+\sum\limits_{|\bj|=1}f^{\bj}(\uy)(\ux-\uy)^{\bj},\,\,\uy\in\partial\Omega\]
is the field of Taylor polynomials of a $C^{1,\nu}$-function in $\R^m$. A classical theorem of Whitney \cite{whitney} shows that such a function $f$ may be extended to a $C^{1,\nu}$-smooth function $\tilde{f}$ in $\R^m$, satisfying moreover that
\[
\|\partial_{\ux}^{{\bj}}\tilde{f}(\ux) \| \le c\,[\dist(\ux,\partial\Omega)]^{\nu-1},\,\,|\bj|=2,\,\,\ux\in\R^m\setminus\partial\Omega.
\]
Finally, if $\partial\Omega$ is $d$-summable and $f$ is assumed to be in $\Li(1+\nu,\partial\Omega)$ with $\nu>\frac{d}{m}$, then it follows from \cite[Lemma 4.1]{AAB} that $\La\tilde{f}$ belongs to $L^p(\Omega)$ with $p=\frac{m-d}{1-\nu}>m$. Consequently, the functions $\mathcal{T}_{\varphi,\psi}^{\dagger}(\La \tilde{f})$ and $\mathcal{T}_\varphi^l(\La\tilde{f})$ are continuous in the whole space $\R^m$ (see \cite[Lemma 4.2]{AAB}).

Summarizing, we are led to the following
\begin{theorem}
Let $f\in\Li(1+\nu,\partial\Omega)$ and let $\partial\Omega$ be $d$-summable  with $\nu>\displaystyle\frac{d}{m}$. Then the jump problem
\begin{equation}\label{JP2fractal}
\begin{cases}
\La F=0,\,&\ux\in\Omega_+\cup\Omega_-,\\
{[F]}^+(\ux)-{[F]}^-(\ux)=f(\ux),\,&\ux\in\partial\Omega,\\
[\mathcal{M}F]^+(\ux)-[\mathcal{M}F]^-=(\mathcal{M}\tilde{f})(\ux),\,&\ux\in\partial\Omega,\\
F(\infty)=(\mathcal{M}F)(\infty)=0,
\end{cases}
\end{equation}
has a solution given by
\begin{equation}\label{SFJP}
F(\ux)=\tilde{f}(\ux)\chi_\Omega(\ux)-\mathcal{T}_{\varphi,\psi}^{\dagger}({\La\tilde{f}})(\ux),
\end{equation}
where $\chi_\Omega$ stands for the characteristic function of $\Omega$.
\end{theorem}
\pf That $F$ satisfies $\La F=0$ may be verified directly on the basis of the formulas  
\[
\mathcal{M}\bigg[\mathcal{T}_{\varphi,\psi}^{\dagger}\La \tilde{f}\bigg]=\mathcal{T}_\phi^l(\La \tilde{f}),\,\,\La[\mathcal{T}_{\varphi,\psi}^\dagger(\La\tilde{f})]=\hD[\mathcal{T}_\varphi^l(\La\tilde{f})]=\begin{cases}\La \tilde{f}\,&\mbox{in}\,\,\Omega_+\\0\,&\mbox{in}\,\,\Omega_-.\end{cases}
\]
The next thing we have to do is to prove that $F$ given by \eqref{SFJP} satisfies the jump conditions in \eqref{JP2fractal}, which is simply deduced from the previously mentioned continuity of $\mathcal{T}_{\varphi,\psi}^{\dagger}(\La\tilde{f})$ and $\mathcal{T}_\varphi^l(\La\tilde{f})$. On the other hand, it is a matter of routine to check the vanishing conditions $F(\infty)=(\mathcal{M}F)(\infty)=0$.\qed  

It seems that the uniqueness of the solution of \eqref{JP2fractal} is closely connected with the removability of $\partial\Omega$ for the class of $\R_{0,m}$-valued continuous functions, i.e. Painleve theorem. Unfortunately, this is not the case for general $d$-summable boundary $\partial\Omega$. Nevertheless, the picture of uniqueness of the jump problem \eqref{JP2fractal} for general $d$-summable boundary can be worked out in much the same way, the only difference being in the usage of a Dolzhenko theorem proved in \cite{ADB} in place of Painleve one. We
close the paper by stating this result in the following theorems, where the symbol $\dim_H(\partial\Omega)$ denotes the Hausdorff dimension of
$\partial\Omega$.
\begin{theorem}
Let be $f\in \Li(1+\nu,\partial\Omega)$, with $\nu>\frac{d}{m}$ and let
$$\dim_H(\partial\Omega)-m+1<\beta<\frac{m\nu-d}{m-d}.$$
Then the function given by \eqref{SFJP} is the unique solution of the jump problem \eqref{JP2fractal} which belongs to the class $$\Li_\beta:=\{g:\mathcal{M}^ig\in\Li(\overline{\Omega_+},\beta)\cap\Li(\overline{\Omega_-},\beta),\;\mathcal{M}^ig(\infty)=0,\;i=0,1\}.$$
\end{theorem}
Following the ideas in \cite{tamayo}, we can arrive at a new version involving the refined Marcinkiewicz exponent:
$$\mathfrak{ m}^*(\partial\Omega):=\inf\{\max\{\mathfrak{m}^+(\partial\Omega,\ux),\mathfrak{m}^-(\partial\Omega,\ux)\},\ux\in\partial\Omega\},$$
where
$$\mathfrak{m}^+(\partial\Omega,\ux)=\left\{p:\lim_{r\to 0}\int_{(B_r(\ux)\setminus\partial\Omega)\cap\Omega_+}\frac{d\uy}{\dist^p(\ux,\partial\Omega)}<\infty\right\},$$
and
$$\mathfrak{m}^-(\partial\Omega,\ux)=\left\{p:\lim_{r\to 0}\int_{(B_r(\ux)\setminus\partial\Omega)\cap\Omega^*}\frac{d\uy}{\dist^p(\ux,\partial\Omega)}<\infty\right\}.$$
Here, $\Omega^*:=\Omega_-\cap\{\ux:\|\ux\|<r\}$, where $r$ is selected such that $\partial\Omega$ is wholly contained inside the ball of radius $r$.
\begin{theorem}
Let $f\in\Li(\partial\Omega,1+\nu)$, with $\nu>1-\frac{\mathfrak{m}^*(\partial\Omega)}{m}$ and
$$\dim_H(\partial\Omega)-m+1<\beta<1-\frac{m(1-\nu)}{\mathfrak{m}^*(\partial\Omega)}.$$
Then the function given by \eqref{SFJP} is the unique solution of the jump problem \eqref{JP2fractal}, such that $F(\ux)$ and $(\mathcal{M}F)(\ux)$ belong to the classes $\Li(\overline{\Omega_+},\beta)$ and $\Li(\overline{\Omega_-},\beta)$.
\end{theorem}
\section*{Acknowledgments}
The author wishes to express his gratitude to the financial support of the Postgraduate Study Fellowship of the  Secretar\'ia de
Ciencia, Humanidades, Tecnolog\'ia e Innovaci\'on (SECIHTI) (grant number 1043969). 

\section*{Declaration of Competing Interest}
The authors have no conflicts of interest to declare. 

\section*{Data availability}
Not applicable.

\section*{ORCID}
\noindent Daniel Alfonso Santiesteban: \url{https://orcid.org/0000-0003-0248-3942}\\
Ricardo Abreu Blaya: \url{
https://orcid.org/0000-0003-1453-7223}\\
Daniel Alpay: \url{https://orcid.org/0000-0002-7612-3598}

\end{document}